\documentclass[11pt,letterpaper]{amsart}

\title[A question of Kwakkel and Markovic]{On a question of Kwakkel and Markovic\\ on existence of wandering domains\\ with bounded geometry}

\author[S. Merenkov]{Sergei Merenkov}
\address{Department of Mathematics, The City College of New York and CUNY Graduate Center, New York, NY 10031, USA}
\email{smerenkov@ccny.cuny.edu}
\thanks{Supported by NSF grant DMS-2247364.}
\subjclass[2020]{}

\usepackage{amsmath, amssymb, amsthm,latexsym}
\usepackage{graphicx}
\newcommand\C{{\mathbb C}}

\newcommand\N{{\mathbb N}}
\newcommand\D{{\mathbb D}}

\newcommand\Z{{\mathbb Z}}

\newcommand\T{{\mathbb T}}

\newcommand\hC{{\hat{\mathbb C}}}

\renewcommand\:{\colon}

\newcommand\Ga{\Gamma}

\newtheorem{theorem}{Theorem}[section]
\newtheorem{definition}[theorem]{Definition}

\newtheorem{lemma}[theorem]{Lemma}

\newtheorem{proposition}[theorem]{Proposition}
\newtheorem{corollary}[theorem]{Corollary}
\newtheorem{question}[theorem]{Question}
\newtheorem{remark}[theorem]{Remark}

\theoremstyle{definition}

\usepackage{hyperref}
\hypersetup{
    colorlinks=true,
    linkcolor=blue,
    filecolor=magenta,      
    urlcolor=cyan,
}

\begin{document}


\abstract{
A question of F.~Kwakkel and V.~Markovic on existence of $\mathcal C^1$-diffeomorphisms of closed surfaces that permute a dense collection of domains with bounded geometry is answered in the negative. In fact, it is proved that for closed surfaces of genus at least one such diffeomorphisms do not exist regardless of whether they have positive or zero topological entropy.
} 
\endabstract

\maketitle


\section{Introduction}\label{s:Intro}

\noindent
In~\cite{KM10}, F.~Kwakkel and V.~Markovic posed the following question (see below for the definitions):

\begin{question}\cite[Question~1, p.~512]{KM10}
{Let $M$ be a closed surface. Do there exist diffeomorphisms $f\in{\rm Diff}^1(M)$ with positive entropy that permute a dense collection of domains with bounded geometry?}
\end{question}

The authors gave a negative answer under the assumption that $f\in{\rm Diff}^{1+\alpha}(M)$ with $\alpha>0$.
In this paper we answer the above question, also in the negative. In fact, we prove that, unless $M$ is the sphere,  there are no such diffeomorphisms regardless of whether the entropy is positive or zero. 

We recall some definitions from~\cite{KM10}. A \emph {closed surface} $M$ is a smooth, closed, oriented Riemannian 2-manifold, equipped with the canonical metric induced from the standard conformal metric of the universal cover $\hC, \C$, or $\D$, the sphere, the plane, or the unit disk, respectively. 
Let $S\subset M$ be compact and $\{D_k\}_{k\in\Z}$ the collection of connected components of the complement of $S$, with the property that ${\rm Int}({\rm Cl}(D_k))=D_k,\ k\in\Z$, where ${\rm Int}$ and ${\rm Cl}$ stand for the interior and the closure,  respectively.

\begin{definition}\label{D:PWD}
We say that a homeomorphism $f$ of $M$ permutes a dense collection of domains $\{D_k\}_{k\in\Z}$ if

\noindent
{\rm (1)} $f(S)=S$ and ${\rm Cl}(D_k)\cap {\rm Cl}(D_{k'})=\emptyset$ if $k\neq k'$,

\noindent
{\rm (2)}  $f^n(D_k)\cap D_k=\emptyset$ for all $k, n\in\Z,\ n\neq0$, and

\noindent
{\rm (3)} $\bigcup_{k\in\Z} D_k$ is dense in $M$.
\end{definition} 

We refer to the invariant set $S$ in Definition~\ref{D:PWD}  as a \emph{residual set} of $f$.

\begin{definition}
A collection of domains $\{D_k\}_{k\in\Z}$ on a surface $M$ is said to have bounded geometry if ${\rm Cl}(D_k),\ k\in\Z$,  are contractible in $M$ 
and there exists a constant $C\ge1$ such that for every domain $D_k$ in the collection, there are $p_k\in D_k$ and $0<r_k\le R_k$ with 
$$
B(p_k, r_k)\subseteq D_k\subseteq B(p_k, R_k)\quad {\rm and}\quad R_k/r_k\le C.
$$ 
Here and in what follows, $B(p,r)$ denotes the open disk in $M$ centered at $p$ of radius $r>0$. 
\end{definition}
Note that if $M$ is a closed surface, bounded geometry of domains in $\{D_k\}_{k\in\Z}$ implies that ${\rm diam}\, D_k\to0$ as  $|k|\to\infty$, where ${\rm diam}$ stands for the diameter.

The  main result of this paper is the following theorem.

\begin{theorem}\label{T:Main}
If $M$ is a closed surface, other than the sphere, then there does not exist $f\in{\rm Diff}^1(M)$ that permutes a dense collection of domains with bounded geometry.
\end{theorem}

Theorem~\ref{T:Main} is false in the case when $M=\hC$, as can be seen by ta\-king $f(z)=\lambda z,\ \lambda>1$. Indeed, one needs to fill the fundamental annulus $\{z\: 1\le |z|<\lambda\}$ with geometric disks and spread them out using the dynamics of $f$. 
In addition, without the bounded geometry assumption the conclusion fails even in the $\mathcal C^{1+\alpha}$ class. This can be seen by choosing two $\mathcal C^{1+\alpha}$ diffeomorphisms  $g, h$ of the unit circle $\T^1$ that are semi-conjugate but not conjugate to irrational rotations, i.e., $g$ and $h$ have wandering intervals, and taking $f=g\times h$, a diffeomorphism of $\T^2$. Such diffeomorphisms $g, h$ exist by~\cite{He79}, and  $f$ permutes the collection $\{I\times J\}$, where $I$ and $J$ are wandering intervals of $g$ and $h$, respectively. 

The paper is organized as follows. In Section~\ref{S:ICS} we demonstrate existence of an invariant conformal structure on the residual  set of a hypothetical $\mathcal C^1$-diffeomorphism $f$ of $M$ that permutes a dense collection of domains with bounded geometry. In Section~\ref{S:TMS} we discuss transboundary modulus of curve families and prove existence and uniqueness of extremal distributions for such a modulus. After a short discussion of the Nielsen--Thurston classification in Section~\ref{S:PA}, we proceed by first eliminating in Section~\ref{S:PAD} the possibility of a hypothetical diffeomorphism $f$ above to be pseudo-Anosov, and then, with the aid of holomorphic quadratic differentials discussed in Section~\ref{S:HQD}, we exclude periodic and reducible diffeomorphisms in the final Section~\ref{S:PR}. 

\medskip
\noindent
{\bf Acknowledgements.}
The author thanks Mario Bonk and Misha Lyubich for fruitful conversations, and the Institute for Mathematical Sciences at Stony Brook University for the hospitality. The author also thanks both anonymous referees for numerous comments and suggestions that helped to substantially improve the presentation, especially that of Section~\ref{S:PAD}.

\section{Invariant Conformal Structure}
\label{S:ICS}

\noindent
A non-constant orientation preserving  homeomorphism  $h\: U \to U'$
between  open  subsets $U$ and $U'$ of Riemann surfaces $M$ and $M'$, respectively, is called \emph{quasiconformal}, or $K$-\emph{quasiconformal}, if the map $h$ is in the Sobolev space $W_{\rm loc}^{1,2}$
and 
$$
K_h(p)=\frac{\max_{|v|=1}|D_ph(v)|}{\min_{|v|=1}|D_ph(v)|}\le K,\quad {\rm a.e.}\  p\in U,
$$
where $|\cdot|$ is the Euclidean norm.
In this case we say that the \emph{dilatation} of $h$ is bounded by $K$.
The assumption $h\in W_{\rm loc}^{1,2}$ means that $h$ and the first distributional partial derivatives of $h$ are locally in the Lebesgue space $L^2$.   
If $S$ is a measurable subset of $U$ and $K_h(p) = 1$ for a.e. $p \in S$, then
the map $h$ is called \emph{conformal} on $S$. 

We will need the following lemma, which we also prove.

\begin{lemma}\cite[Lemma~2.3]{KM10}\label{L:UQC}
Let $M$ be a closed surface and let $f\in{\rm Diff}^1(M)$ permute a dense collection of domains $\{D_k\}_{k\in\Z}$. If the domains in $\{D_k\}_{k\in\Z}$ have bounded geometry, then $f^n,\ n\in\Z$, have uniformly bounded  dilatations on $S=M\setminus\cup_{k\in\Z} D_k$. I.e., $K_{f^n}(p)$ is   bounded above by a constant independent of $n\in\Z$ and $p\in S$. 
\end{lemma}

\begin{remark}
Lemma~\ref{L:UQC} is the only result where the assumption that $f$ is a $\mathcal C^1$-diffeomorphism is needed to prove the main Theorem~\ref{T:Main}. Once Lemma~\ref{L:UQC} is proved below, the rest of the arguments hold true for an arbitrary quasiconformal $f$.
\end{remark}

\begin{proof}[Proof of Lemma~\ref{L:UQC}] 
Let $\widetilde M$ denote the universal cover of $M$ and $\psi\: \widetilde M\to M$ be the projection map. 
Since $\psi$ is a local isometry, below we make no distinction between  $D_k,\ k\in\Z$, that have sufficiently small diameters and their lifts to $\widetilde{M}$ and, whenever convenient, identify the map $f$ with its lift $\tilde f$  to $\widetilde{M}$ under the map  $\psi$. 
 
Let $n\in\Z$ be fixed.
Since $S$ is nowhere dense,  ${\rm Int}({\rm Cl}(D_k))=D_k,\ k\in\Z$, and the closures ${\rm Cl}(D_k),\ k\in\Z$, are pairwise disjoint,  and ${\rm diam}\, D_k \to 0$ as $|k| \to \infty$, which follows from the bounded
geometry assumption, for each $p\in S$ there exists a sequence of complementary components $(D_{k_i})_{i\in\N}$ of $S$ that accumulate at $p$, i.e., 
$$
 d_{\rm Hausd}(\overline{D_{k_i}},\{p\})\to0\  {\rm as}\  i\to\infty,
$$  
where  $d_{\rm Hausd}$ denotes the Hausdorff distance.
Also, since $f$ is $\mathcal C^1$-dif\-fe\-ren\-ti\-able on a compact surface $M$,  we have
\begin{equation}\label{E:Diff}
f^n(q+x)-f^n(q)-D_qf^n(x)=o(x),\quad x\to0,
\end{equation}
where $|o(x)|/|x|\to0$ as $|x|\to0$, uniformly in $q$. Note that, as discussed in the first paragraph of this proof, here we abused the notation and wrote $f^n$ instead of $\psi^{-1}\circ f^n\circ\psi$ with consistent choice of the branch of $\psi^{-1}$, and $q$ and $x$ are elements of the universal cover, i.e., $\D, \C$, or $\C\cup\{\infty\}$.  

The domains $D_{k_i},\ i\in\N$, having bounded geometry means that the\-re exist $C\geq1, p_i\in D_{k_i}$, and $0<r_i\leq R_i,\ i\in \N$, with
$$
B(p_i, r_i)\subseteq D_{k_i}\subseteq B(p_i, R_i),\quad {\rm and}\quad {R_i}/{r_i}\leq C.
$$
Thus, by possibly passing to a subsequence, we may assume that the Hausdorff limit of the rescaled domains 
$$
\left((D_{k_i}-p_{i})/r_{i}\right)_{i\in\N}
$$ 
is a closed set $\Omega$ such that 
\begin{equation}\label{E:C1}
\overline{B(0,1)}\subseteq\Omega\subseteq \overline{B(0, C)}.
\end{equation}
Again, here we identified $p_i$ and $D_{k_i},\ i\in\N$, with the corresponding lifts under the projection map $\psi$.
 Let $D_{k_i}',\ i\in\N$,  denote the image of $D_{k_i},\ i\in\N$, under the map $f^n$. 
Because $f^n$ is in ${\rm Diff}^1(M)$, and hence is locally bi-Lipschitz, after possibly passing to yet a further subsequence, we may assume that the sequence 
$$
\left((D_{k_i}'-f^n(p_{i}))/r_{i}\right)_{i\in\N}
$$ 
Hausdorff converges to a closed set $\Omega'$. 
Similar to~\eqref{E:C1}, we know that there exists $p'\in\Omega'$ and $r'>0$ such that 
\begin{equation}\label{E:C2}
\overline{B(p',r')}\subseteq\Omega'\subseteq \overline{B(p',Cr')},
\end{equation}
where the constant $C\ge1$ is the same as above.
Since Equation~\eqref{E:Diff} is uniform in $q$, in particular it  holds for $q=p_{i},\ q+x\in D_{k_i},\ i\in\N$, one concludes that 
\begin{equation}\label{E:C3}
\Omega'=D_pf^n(\Omega). 
\end{equation}
Combining~\eqref{E:C1}, \eqref{E:C2}, and \eqref{E:C3}, and using the fact that the constant $C$ does not depend on $p\in S$ or  $n\in\Z$, we conclude that the maps $f^n,\ n\in\Z$, have uniformly bounded dilatations on $S$, as claimed.
\end{proof}

In the proof of Theorem~\ref{T:Main} we use a modification, Proposition~\ref{P:Tukia} below, of a result by D.~Sullivan~\cite{Su78} or P.~Tukia~\cite{Tu80}. 
A \emph{Beltrami form} on a measurable subset $S$ of a Riemann surface $M$ is a measurable $(-1,1)$-form on $S$ given in a local chart by $\mu=\mu(z)d\bar z/dz$, where we assume that $\|\mu\|_\infty<1$; see, e.g., \cite{BF14} for background on Beltrami forms. The measurable function $z\mapsto \mu(z)$ is called a \emph{Beltrami coefficient} of $\mu$. It is dependent on a local chart.
If $h\: U\to U'$ is a quasiconformal map between open subsets $U$ and $U'$ of Riemann surfaces $M$ and $M'$, respectively, and $\mu$ is a Beltrami form on $U'$, the \emph{pullback} of $\mu$ under $h$ is the Beltrami form given in local charts by
$$
h^*(\mu)=\left(\frac{h_{\bar z}+\mu(h)\overline{h_z}}{h_{z}+\mu(h)\overline{h_{\bar z}}}\right)\frac{d\bar z}{dz}.
$$ 

The pullback of the 0 Beltrami form is denoted by $\mu_h$.
We say that a Beltrami form $\mu$ on $S$ is  \emph{invariant} under $h$
if $S$ is \emph{invariant} under $h$, i.e., $S\subseteq U=U',\ h(S)=S$ up to a measure zero set, and
$
\mu=h^*(\mu).
$ 
In particular, the 0 Beltrami form is invariant under $h$ if and only if $h$ is conformal on $S$.

\begin{proposition}\label{P:Tukia}
Let $M$ be a Riemann surface and $S\subseteq M$ be a measurable subset. Let $H$ be a countable group of quasiconformal maps $h$ of $M$ that leave $S$ invariant  and that have uniformly bounded dilatation on $S$. Then there exists a Beltrami form $\mu$ on $S$ that is invariant under each map $h\in H$. 
\end{proposition}
\begin{proof}
We may assume that $S$ has positive measure, for otherwise there is nothing to prove.

For arbitrary quasiconformal maps $f, g$, expressing the Beltrami coefficient $\mu_{g\circ f}$ in terms of Beltrami coefficients $\mu_g$ and $\mu_f$ in some local charts, we have
\begin{equation}\label{E:Composition}
\mu_{g\circ f}=f^*(\mu_g)=T_{f}(\mu_g(f)),
\end{equation}
where
$$
T_{f}(w)=\frac{\mu_f+\frac{\overline{f_z}}{f_z}w}{1+\overline{\mu_f}\frac{\overline{f_z}}{f_z}w}
$$
is an isometry of the disk model hyperbolic plane $\D$ for a.e.\  $z$. Namely, $T_f$ is a hyperbolic isometry in $w$ for each fixed $z$ where $T_f$ is defined, i.e., at each point  of differentiability of $f$ with $f_z\neq0$. Such points $z$ form a set of full measure.

We consider
$$
B(z)=\{\mu_{h}(z)\:\ h\in H\},
$$
a subset of $\D$ defined for a.e.\ $z\in  S$ up to a rotation which depends on a local chart, since if $\zeta=\zeta(z)$ is a conformal change of local charts, $\mu_h(z)=\mu_h(\zeta(z))\frac{\overline{\zeta'(z)}}{\zeta'(z)}$, i.e., the value of $\mu_h$ is multiplied by a complex number of modulus 1, a rotation. 
Since $H$ is a group, Equation~\eqref{E:Composition} gives 
\begin{equation}\label{E:Transform}
\begin{aligned}
T_{h}(B(h(z)))&=\left\{T_h(\mu_{h'}(h(z)))\:h'\in H\right\}\\&=\left\{\mu_{h'\circ h}(z)\:h'\in H\right\}=B(z),
\end{aligned}
\end{equation}
for each $h\in H$, a.e.\ $z\in S$, and an appropriate choice of local charts $z$ and $h(z)$.
If $B$ is a non-empty bounded subset of $\D$, we define by $P(B)$ to be the center of the smallest closed hyperbolic disk that contains $B$. Elementary hyperbolic geometry  shows that $P(B)$ is well defined.
Moreover, $P$ satisfies  (see~\cite{Tu80} for details): 

\noindent
1) if $B$ is contained in a hyperbolic disk centered at 0 with radius $\delta$ in the hyperbolic metric of $\D$, then $P(B)$ is at most hyperbolic distance $\delta$ from the origin in $\D$;

\noindent
2) the map $P$ commutes with each hyperbolic isometry of $\D$, i.e., 
$$P(T(B))=T(P(B))$$ for any hyperbolic isometry $T$ of $\D$ and arbitrary non-empty bounded subset $B$ of $\D$.  

We now define, for a.e.\ $z\in S$, 
$$
\mu(z)=
P(B(z)).
$$
This definition, even though dependent on local charts, gives rise to a global Beltrami form $\mu$ on $M$ because $P$ commutes with hyperbolic isometries, and in particular with rotations.
The assumption that the elements of $H$ are $K$-quasiconformal on $S$ with the same constant $K$ and property 1) above give that $\mu$ is an essentially bounded Beltrami form.
Equation~\eqref{E:Transform} combined with the commutative  property 2) give, in appropriate local charts,
$$
\mu(z)=P(B(z))=P(T_{h}(B(h(z))))=T_{h}(P(B(h(z))))=T_{h}(\mu(h(z))),
$$
for each $h\in H$ and a.e.\ $z$ in $S$. 
By~\eqref{E:Composition}, this is equivalent to 
$$
\mu=h^*(\mu),
$$
on $S$ for each $h\in H$, i.e., $\mu$ is invariant under each $h\in H$.
%
%
%
\end{proof}

An alternative proof of Proposition~\ref{P:Tukia} uses the barycenter of the convex hull to define $\mu$ rather than the center of the smallest hyperbolic disk as above. This approach was used in~\cite{Su78}.

An application of the Measurable Riemann Mapping Theorem~\cite[Theorem~1.27]{BF14} gives that for the Beltrami form $\mu$ from Proposition~\ref{P:Tukia}, extended by 0 outside $S$,  there exists a Riemann surface $M'$ and a quasiconformal homeomorphism $\phi\: M\to M'$ such that $\mu=\mu_\phi$. 
Since conformality of a quasiconformal map $h$ of $M$ on an invariant measurable subset $S$ is equivalent to the 0 Beltrami form being invariant under $h$, we get the following corollary to Proposition~\ref{P:Tukia}.
\begin{corollary}\label{C:InvConf}
Let $M$ be a Riemann surface and $S\subseteq M$ be a measurable subset. Let $H$ be a countable group of quasiconformal maps $h$ of $M$ that leave $S$ invariant  and that have uniformly bounded dilatation on $S$. Then there exists a quasiconformal map $\phi\: M\to M'$ of $M$ onto another Riemann surface $M'$ such that the  conjugate family $H'=\{\phi\circ h\circ \phi^{-1}\: h\in H\}$ consists of quasiconformal maps of $M'$ that leave  $S'=\phi(S)$ invariant and are conformal on it.
\end{corollary}
In particular, we have the following result.
\begin{corollary}\label{C:ConfOnS} If $M$ is a closed Riemann surface and 
 $f\in{\rm Diff}^1(M)$ permutes a dense collection of domains with bounded geometry, then  there exists a quasiconformal map $\phi\: M\to M'$ of $M$ onto another Riemann surface $M'$ such that for each $n\in\Z$, the maps $\phi\circ f^n\circ \phi^{-1}$ are all quasiconformal maps that are conformal on $S'=\phi(S)$, where $S$ is the residual  set of $f$.
\end{corollary}
\begin{proof}
We first apply Lemma~\ref{L:UQC} to conclude that $f^n,\ n\in\Z$, have uniformly bounded dilatations on the residual set $S$ of $f$, and then  apply  Corollary~\ref{C:InvConf} to the residual set $S$ of $f$ and the cyclic group $f^n,\ n\in\Z$.
\end{proof}
\begin{remark}\label{R:QCC}
Therefore, in what follows, we may and will assume that $f$ is a quasiconformal homeomorphism of $M$ that is conformal on its residual set $S$.
\end{remark}

\section{Transboundary modulus}\label{S:TMS}

\noindent
The  following notion of modulus is inspired by O.~Schramm's transboundary extremal length, introduced in~\cite{Sch95}. An adaptation in the setting of Sierpi\'nski carpets has been used extensively, notably in~\cite{Bo11, BM13}.

Let $M$ be a closed surface  and $\sigma$ the area measure on $M$. If $\gamma\: I\to M$ is a rectifiable curve, where $I$ is a closed interval, then $\gamma$ factors as $\gamma=\gamma_s\circ s_\gamma$, where $s_\gamma\: I\to [0,l(\gamma)]$ is the associated \emph{length function}, with $l(\gamma)$ being the length of $\gamma$, and $\gamma_s\: [0,l(\gamma)]\to M$ is the unique 1-Lipschitz continuous map, the \emph{arc length parametrization} of $\gamma$; see~\cite{Va71} for more details. For a non-negative Borel measurable function $\rho$ on $M$ and a rectifiable curve $\gamma$ as above, $\int_\gamma\rho\, ds$ is defined to be $\int_0^{l(\gamma)}\rho(\gamma_s(t)) dt$. If $\gamma$ is only locally rectifiable, we define $\int_\gamma\rho\,  ds=\sup\int_{\gamma'}\rho ds$, where the supremum is taken over all rectifiable subcurves $\gamma'$ of $\gamma$.

By a \emph{mass density} we mean a non-negative Borel measurable function $\rho$ on $M$. If $\Gamma$ is a curve family in $M$, a mass density $\rho$ is called \emph{admissible} for $\Gamma$ if
$$
\int_\gamma\rho\, ds\ge1,\quad {\rm for\ all\ locally\ rectifiable}\ \gamma\in\Gamma,
$$
where $ds$ denotes the arc-length element.
The $2$-\emph{modulus} of $\Gamma$ is defined as
$$
{\rm mod}_2(\Gamma)=\inf\left\{\int_{M} \rho^2 d\sigma\right\},
$$
where the infimum is taken over all $\rho$ that are admissible for $\Gamma$. 

Let $S\subset M$ be a compact set and $M\setminus S=\cup_{k\in \Z}D_k$, where $\{D_k\}_{k\in \Z}$ is a non-empty collection of disjoint complementary components of $S$.
If
$\Gamma$ is  a family of curves in
$M$, then we  define  the \emph{transboundary modulus} of $\Gamma$ with respect to $S$, denoted by ${\rm mod}_{S}(\Gamma)$,  as follows. Let  $\rho$ be a \emph{mass distribution for} $S$, i.e., a function $\rho\colon S\cup\{D_k\: k\in\Z\}\to[0,\infty]$, where 
the restriction $\rho|_S$ is a  non-negative Borel measurable function, and $\rho(D_k),\ k\in\Z$, is a non-negative number for each $k\in \Z$. We refer to $\rho|_{S}$ as the \emph{continuous part} of a mass distribution $\rho$ and to $\{\rho(D_k)\}_{k\in \Z}$ as the \emph{discrete part}.
A mass distribution $\rho$ is \emph{admissible} for
$\Gamma$ if  
there exists  $\Gamma_0\subseteq \Gamma,\ {\rm mod}_2(\Gamma_0)=0$, called an \emph{exceptional} curve family, such that   
\begin{equation}\label{E:Adm}
\int_{\gamma\cap S}\rho\, ds+\sum_{k\: \gamma\cap D_k\neq\emptyset}\rho(D_k)\ge1,
\end{equation}
for  each 
$\gamma\in\Gamma\setminus \Gamma_0$. Here, if $\gamma$ is rectifiable, 
$$
\int_{\gamma\cap S}\rho\, ds=\int_{[0,l(\gamma)]\cap \gamma_s^{-1}(S)}\rho(\gamma_s(t))dt,
$$
and extended as above to locally rectifiable curves. Note that non-locally rectifiable curves have 2-modulus 0, and therefore excluded.
We set 
\begin{equation}\label{E:Modulus}
{\rm mod}_{S}(\Gamma)=\inf\left\{{\rm mass}(\rho)\right\},
\end{equation}
where the infimum is taken over all mass distributions $\rho$ that 
are admissible for $\Gamma$, and 
\begin{equation}\label{E:Mass}
{\rm mass}(\rho)=\int_{S} \rho^2 d\sigma+\sum_{k\in \Z}\rho(D_k)^2
\end{equation}
is the \emph{total mass} of $\rho$.
If the set $S$ has measure zero, then we ignore the continuous part of the distribution $\rho$ in~\eqref{E:Adm} and~\eqref{E:Mass}. 
Excluding an exceptional curve family $\Ga_0$  ensures that in many interesting cases
and for some relevant  curve families $\Ga$
an admissible mass distribution exists and  $0<{\rm mod}_S(\Gamma)<\infty$. 

We think of a mass distribution $\rho$ as an element of the direct sum Hilbert space $L^2\oplus l^2$, where $L^2=L^2(S)$, and $l^2$ consists of all sequences $(\rho_k)_{k\in \Z}$ with $\sum_{k}\rho_k^2<\infty$. The norm of $\rho\in L^2\oplus l^2$ is given by
$
\|\rho\|_{L^2\oplus l^2}=\left(\|\rho\|_{L^2}^2+\|\rho\|_{l^2}^2\right)^{1/2}
$, where
$$
\|\rho\|_{L^2}=\left(\int_{S}\rho^2 d\sigma\right)^{1/2},\quad \|\rho\|_{l^2}=\left(\sum_{k\in \Z}\rho_k^2\right)^{1/2}.
$$

Transboundary modulus has some familiar properties of the $2$-mo\-du\-lus, with almost identical proofs, e.g., quasi-invariance under quasiconformal maps, monotonicity, and subadditivity; see~\cite{Va71} for classical  background. \emph{Quasi-invariance} is the property that if $h\: M\to M'$ is a quasiconformal map between two Riemann surfaces, then there exists a constant $C\ge1$ such that  for every curve family $\Gamma$ in $M$ one has
$$
\frac{1}{C}\,{\rm mod}_S(\Gamma)\le {\rm mod}_{S'}(\Gamma')\le C\, {\rm mod}_S(\Gamma),
$$
where $S'=h(S)$ and $\Gamma'=h(\Gamma)=\{h(\gamma)\: \gamma\in\Gamma\}$. In fact, if $h$ is $K$-quasiconformal, then one can choose $C=K$ above. If $h$ is conformal on $S$, the modulus is invariant, i.e., $C=1$; see Lemma~\ref{L:Confinv} below.
 \emph{Monotonicity} means that if $\Gamma'\subseteq\Gamma$, then
$$
{\rm mod}_S(\Gamma')\le{\rm mod}_S(\Gamma).
$$
\emph{Subadditivity} says that if $\Gamma=\cup_{i\in\N}\Gamma_i$, then
$$
{\rm mod}_S(\Gamma)\le\sum_{i\in\N}{\rm mod}(\Gamma_i).
$$    

The following proposition is the main result of this section.
 \begin{proposition}\label{P:Extremaldistr}
 If $\Gamma$ is a curve family with ${\rm mod}_S(\Gamma)<\infty$, and if
 $\{D_k\}_{k\in \Z}$ have bounded geometry, then an extremal distribution $\rho$, i.e., an admissible distribution for which the infimum in the definition~\eqref{E:Modulus} is achieved, exists and is unique as an element of $L^2\oplus l^2$.
\end{proposition}
The proof of this proposition below mimics that of~\cite[Proposition~2.4]{BM13} and requires auxiliary results. 

\subsection{Auxiliary lemmas}
As above, let $M$ be a closed surface and $\sigma$ its area measure. 
To prove Proposition~\ref{P:Extremaldistr} we need Corollary~\ref{C:Modcomp0} below, 
which, in turn, is derived from the following two results.

\begin{lemma}\cite{Bo88}, \cite[Lemma~2.2]{BM13}\label{L:Boj}
Let $\{B_k\}_{k\in I}$ be a finite or countable collection of disjoint geodesic disks in a closed surface $M$. Further,  let $\{a_k\}_{k\in I}$ be non-negative numbers. Then, for each $\lambda\ge1$ there exists $C\ge0$ such that
\begin{equation}\label{E:Boj}
\left\|\sum_{k\in I}a_k\chi_{\lambda B_k}\right\|_{L^2}\leq C \left\|\sum_{k\in I}a_k\chi_{B_k}\right\|_{L^2},
\end{equation}
where $\lambda B_k$ denotes a disk concentric with $B_k$ and whose radius is $\lambda$ times that of $B_k$.
\end{lemma}
\begin{remark}
In the above cited sources the lemma is proved in the planar or spherical settings, but  it extends to arbitrary closed surfaces with an almost identical proof using an uncentered maximal function.
\end{remark}

\begin{lemma}\label{L:Modcomp}
Let $M$ be a closed surface and $S\subset M$ be a compact subset such that $S=M\setminus\cup_{k\in\Z} D_k$, where the domains $D_k,\ k\in \Z$,  have bounded geometry. If $\delta>0$, then there exists a constant $C>0$ such that for any curve family $\Gamma$  in  $M$ consisting of curves $\gamma$ with ${\rm diam}(\gamma)\ge\delta$,  we have
$$
{\rm mod}_{2}(\Gamma)\le C\cdot{\rm mod}_{S}(\Gamma).
$$ 
\end{lemma}
\begin{proof}
We assume that ${\rm mod}_{S}(\Gamma)$ is finite for otherwise there is nothing to prove.
Let $\epsilon>0$ be arbitrary.
There exists a mass distribution $\rho$ with ${\rm mass}(\rho)<{\rm mod}_{S}(\Gamma)+\epsilon$, and an exceptional curve family $\Gamma_0\subseteq\Gamma$ such that
$$
\int_{\gamma\cap S}\rho\, ds+\sum_{k\:\gamma\cap D_k\neq\emptyset}\rho(D_k)\ge1,
$$ 
for  all $\gamma\in\Gamma\setminus\Gamma_0$.


Recall, the bounded geometry assumption is, for each $k\in\Z$, the existence of disks $B(p_k, r_k), B(p_k, R_k)$ such that 
$$
B(p_k, r_k)\subseteq D_k\subseteq B(p_k, R_k),\quad {\rm with}\quad {R_k}/{r_k}\leq C.
$$
Let $r=\sup\{r_{k}\}>0$, a constant that depends on $S$.
We define a mass density $\tilde\rho$ on $M$  by the formula
$$
\tilde\rho=\rho+\sum_{k\in \Z}\frac{\rho(D_k)}{r_k}\chi_{B(p_k,2R_k)},
$$
where the first summand is extended  to the complement of $S$ as 0. Let $\gamma\in\Gamma\setminus\Gamma_0$ be an arbitrary locally rectifible curve. If $\gamma$ intersects $D_k,\ k\in\Z$, and is completely contained in $B(p_k, 2R_k)$, then 
$$
\int_\gamma\chi_{B(p_k,2R_k)}ds\ge \delta.
$$
Otherwise, i.e., if $\gamma$ is not completely contained in $B(p_k,2R_k)$, it  intersects both complementary components of $B(p_k,2R_k)\setminus \overline{B(p_k, R_k)}$, and thus 
$$
\int_\gamma\chi_{B(p_k,2R_k)}ds\ge R_k.
$$
Therefore,
$$
\int_\gamma\tilde\rho\, ds\ge \int_{\gamma\cap S}\rho\, ds+\min\left\{\frac{\delta}{r},1\right\}\cdot\sum_{k\:\gamma\cap D_k\neq\emptyset}\rho(D_k)\ge\min\left\{\frac{\delta}{r}, 1\right\},
$$
i.e., $\tilde\rho/\min\left\{\frac{\delta}{r}, 1\right\}$ is admissible for $\Gamma\setminus\Gamma_0$.
Also, using Lemma~\ref{L:Boj} we conclude
\begin{equation}\label{E:finitemass}
\begin{aligned}
\int_{M}\tilde\rho^2d\sigma&\leq 2\left(\int_S \rho^2 d\sigma+C^2\sum_{k\in \Z}\frac{\rho(D_k)^2}{r_k^2}\sigma\left(B(p_k, r_k)\right)\right)\\&\leq C_1\cdot{\rm mass}(\rho)<C_1\left({\rm mod}_S(\Gamma)+\epsilon\right),
\end{aligned}
\end{equation}
where $C_1$ is a constant that depends on the surface $M$ and the set $S$. Thus, 
$$
\begin{aligned}
{\rm mod}_2(\Gamma)&\le{\rm mod}_2(\Gamma\setminus\Gamma_{0})+{\rm mod}_2(\Gamma_0)\\ &={\rm mod}_2(\Gamma\setminus\Gamma_{0})\le C_2\left({\rm mod}_S(\Gamma)+\epsilon\right),
\end{aligned}
$$
for $C_2=C_1/\left(\min\left\{\frac{\delta}{r}, 1\right\}\right)^2$.
Since $\epsilon$ is arbitrary, we conclude that 
$$
{\rm mod}_2(\Gamma)\le C_2\cdot {\rm mod}_S(\Gamma).$$ 
\end{proof}
\begin{remark}
The inequality in the previous lemma fails without the lower bound assumption on diameters. Indeed, it is elementary to produce a fa\-mi\-ly of curves $\Gamma$, all contained in one complementary domain $D_k$, such that  ${\rm mod}_2(\Gamma)=\infty$. However, for such family one would have ${\rm mod}_S(\Gamma)=1$.
\end{remark}

\begin{corollary}\label{C:Modcomp0}
Let $M$ be a closed surface and $S\subset M$ be a compact subset such that $S=M\setminus\cup_{k\in\Z} D_k$, where the domains $D_k,\ k\in \Z$,  have bounded geometry. If $\Gamma$  is a curve family in  $M$ that does not contain constant curves, then ${\rm mod}_{S}(\Gamma)=0$ implies ${\rm mod}_{2}(\Gamma)=0$.
\end{corollary}
\begin{proof}
Since $\Gamma$ does not contain constant curves, we have
$$
\Gamma=\cup_{i\in\N}\Gamma_i,\quad{\rm where}\quad \Gamma_i=\{\gamma\in\Gamma:\ {\rm diam}(\gamma)\ge1/i\}.
$$
Monotonicity of ${\rm mod}_S$ gives that  ${\rm mod}_S(\Gamma_i)=0$. 
Lemma~\ref{C:Modcomp0} then gives ${\rm mod}_2(\Gamma_i)=0$, and the
subadditivity of ${\rm mod}_2$
implies ${\rm mod}_2(\Gamma)=0$.
\end{proof}

\subsection{Proof of Proposition~\ref{P:Extremaldistr}} Let $(\rho_i)_{i\in\N}$ be a minimizing sequence of admissible mass distributions for $\Gamma$, i.e.,
$$
{\rm mass}(\rho_i)\to{\rm mod}_S(\Gamma),\quad i\to\infty.
$$
The condition ${\rm mod}_S(\Gamma)<\infty$ allows us to assume that ${\rm mass}(\rho_i)\le C$ for some constant $C>0$ and all $i\in\N$. 
By the Banach--Alaoglu Theorem, there exists a subsequence of $(\rho_i)_{i\in\N}$ that weakly converges to $\rho\in L^2\oplus l^2$.
In particular,  we may assume that the limits
$$
\rho(D_k)=\lim_{i\to\infty}\rho_i(D_k)
$$
exist for all $k\in \Z$. For the continuous part of the mass distribution we choose a non-negative Borel measurable representative of the $L^2$-summand of $\rho$. We claim that such a representative along with $\rho(D_k),\ k\in\Z$, is an extremal mass distribution. 

One inequality, namely ${\rm mass}(\rho)\leq{\rm mod}_S(\Gamma)$ follows from the weak lower semicontinuity of norms.
To show the reverse inequality, one needs to demonstrate that $\rho$ is admissible for $\Gamma$.
Since the sequence $(\rho_i)_{i\in\N}$ converges weakly to $\rho\in L^2\oplus l^2$, Mazur's Lemma~\cite[Th.\ 2, p.\ 120]{Yo80} gives that there is a sequence $(\tilde\rho_K)_{K\in\N}$ of convex combinations,
$$
\tilde\rho_K=\sum_{i=1}^K\lambda_{K,i}\rho_i,
$$
that strongly converges to $\rho$ in $L^2\oplus l^2$. 
Since each $\tilde\rho_K,\ K\in\N$, is trivially admissible for $\Gamma$, the sequence $(\tilde\rho_K)_{K\in\N}$ is also a minimizing sequence for  ${\rm mod}_S(\Gamma)$.  
An exceptional curve family $\widetilde\Gamma_K$ for $\tilde\rho_K$ is the union of exceptional curve families for $\rho_i,\ i=1,2,\dots, K$.
 By possibly passing to a subsequence, we may assume that
 $$
 \left\|\tilde\rho_K-\rho\right\|_{L^2\oplus l^2}\le\frac{1}{2^K},\quad K\in\N. 
 $$
 Let 
 $$
 \Gamma_K'=\left\{\gamma\in\Gamma\: \int_{\gamma\cap S}|\tilde\rho_K-\rho|\, ds+\sum_{k\: \gamma\cap D_k\neq\emptyset}|\tilde\rho_K(D_k)-\rho(D_k)|\ge\frac{1}{K}\right\},
 $$
 and 
 $$
 \Gamma'=\cap_{i\in\N}\cup_{K\ge i}\Gamma'_K.
 $$
From this definition we have that if $\gamma\in\Gamma\setminus\Gamma'$, then
 $$
 \lim_{K\to\infty}\left(\int_{\gamma\cap S}|\tilde\rho_K-\rho|\, ds+\sum_{k\: \gamma\cap D_k\neq\emptyset}|\tilde\rho_K(D_k)-\rho(D_k)|\right)=0.
 $$
 The mass distributions
 $$
 \rho_i'=\sum_{K=i}^\infty K|\tilde\rho_K-\rho|,\quad i\in\N,
 $$
 are admissible for $\Gamma'$, and since $\left\|\tilde\rho_K-\rho\right\|_{L^2\oplus l^2}\le{1}/{2^K}$, we have that 
 $${\rm mass}(\rho_i')\to0,\ i\to\infty.
 $$ Thus, ${\rm mod}_S(\Gamma')=0$. Corollary~\ref{C:Modcomp0} then gives ${\rm mod}_2(\Gamma')=0$. We conclude that $\Gamma_\infty=\Gamma'\cup\left(\cup_{K\in\N}\widetilde\Gamma_K\right)$ is an exceptional family for $\rho$. Indeed, subadditivity gives  ${\rm mod}_2(\Gamma_\infty)=0$, and for each $\gamma\in\Gamma\setminus\Gamma_\infty$, since each $\tilde\rho_K$ is admissible,
 $$
 \begin{aligned}
 &\int_{\gamma\cap S}\rho\, ds+\sum_{k\:\gamma\cap D_k\neq\emptyset}\rho(D_k)\\
 &\ge1-\lim_{K\to\infty}\left(\int_{\gamma\cap S}|\tilde\rho_K-\rho|\, ds+\sum_{k\: \gamma\cap D_k\neq\emptyset}|\tilde\rho_K(D_k)-\rho(D_k)|\right)=1.
 \end{aligned}
 $$
This completes the proof of existence of an extremal mass distribution for ${\rm mod}_S(\Gamma)$. 

A simple convexity argument implies that an extremal distribution $\rho$ is also unique. Indeed, if there exists another, different, extremal mass distribution $\rho'$, then
 the average $(\rho+\rho')/2$ is also admissible for $\Gamma$. The strict convexity of $L^2\oplus l^2$, now gives that the average has a strictly smaller total mass, a contradiction. 
\qed

\smallskip

To finish this section, we prove the following conformal invariance  lemma, Lemma~\ref{L:Confinv} below, that makes transboundary modulus particularly useful. Let $f\: M\to M'$ be a homeomorphism between Riemann surfaces that
takes a compact set $S=M\setminus\cup_{k\in\Z} D_k$ onto $S'=M'\setminus\cup_{k\in\Z}D_k'$.
Assume that $f$ is conformal on $S$, and $\rho$ is a mass distribution for $S'$. Let 
$$
L_f(p)=\limsup_{q\to p}\frac{|f(q)-f(p)|}{|q-p|},
$$
where $|\cdot|$ now stands for the intrinsic distance.
This is a non-negative Borel measurable function with $L_f(p)=|D_pf|$ at each point $p$ of differentiability of $f$~\cite[\S5]{Va71}.
We define the \emph{pullback distribution} $f^*(\rho)$ to be $\rho(f)L_f$ on $S$ and $f^*(\rho)(D_k)=\rho(f(D_k)),\ k\in \Z$. 
\begin{lemma}\label{L:Confinv}
Let $f\: M\to M'$ be a quasiconformal homeomorphism that is conformal on $S=M\setminus\cup_{k\in\Z}D_k$.  Let $\Gamma$ be an arbitrary curve fa\-mi\-ly in $M$. If $\rho$ is the extremal mass distribution for $f(\Gamma)=\{f(\gamma)\: \gamma\in\Gamma\}$, then the pullback distribution $f^*(\rho)$ is the extremal mass distribution for $\Gamma$. 
Moreover, 
$$
{\rm mod}_{f(S)}(f(\Gamma))={\rm mod}_S(\Gamma).
$$ 
\end{lemma}
\begin{proof}
Let $\Gamma_0\subseteq f(\Gamma)$ be an exceptional curve family for $\rho$, and let $\Gamma_0'=f^{-1}(\Gamma_0)$ be the preimage subfamily of $\Gamma$. Then, since $f$ is quasiconformal, the geometric definition of quasiconformality gives ${\rm mod}_2(\Gamma_0')=0$; see~\cite{He01} or \cite{Va71}. 
In addition, let $\Gamma_0''$ be the subfamily of $\Gamma$ consisting of all locally rectifiable curves which have closed subcurves on which $f$ is not absolutely continuous. Fuglede's Theorem~\cite[Theorem~28.2]{Va71} gives ${\rm mod}_2(\Gamma_0'')=0$, and thus ${\rm mod}_2(\Gamma_0'\cup \Gamma_0'')=0$. If $\gamma\in\Gamma\setminus(\Gamma_0'\cup\Gamma_0'')$, then, applying~\cite[Theorem~5.3]{Va71},
$$
\begin{aligned}
&\int_{\gamma\cap S}\rho(f)L_f\, ds+ \sum_{k\: \gamma\cap D_k\neq\emptyset}\rho(f(D_k))
\\
&\ge\int_{f(\gamma)\cap S}\rho\, ds+\sum_{k\: f(\gamma)\cap D_k\neq\emptyset}\rho(D_k)\ge1,
\end{aligned}
$$ 
i.e., the pullback mass distribution $f^*(\rho)$ is admissible for $\Gamma$. The total mass of $f^*(\rho)$ is
$$
\begin{aligned}
&\int_S\rho(f)^2L_f^2d\sigma+\sum_{k\in \Z}\rho(f(D_k))^2\\
&=\int_S\rho(f)^2|Df|^2d\sigma+\sum_{k\in \Z}\rho(f(D_k))^2\\
&\le\int_S\rho(f)^2|J_f|d\sigma+\sum_{k\in \Z}\rho(f(D_k))^2\\
&=\int_S\rho^2d\sigma+\sum_{k\in \Z}\rho(D_k)^2={\rm mass}(\rho),
\end{aligned}
$$
where $J_f$ is the Jacobian determinant of $f$, and the above inequality  uses the assumption that $f$ is conformal on $S$. Therefore, 
$$
{\rm mod}_S(\Gamma)\le{\rm mod}_S(f(\Gamma)).
$$
The converse inequality, and indeed the full lemma, follow from the facts that the inverse of a quasiconformal map is quasiconformal, and the inverse of a conformal map is conformal.  
\end{proof}

\section{Nielsen--Thurston Classification}
\label{S:PA}

\noindent
Recall that the \emph{mapping class group} of a closed surface $M$, denoted ${\rm Mod}(M)$, is the group of isotopy classes of  orientation preserving homeomorphisms ${\rm Homeo}^+(M)$. Equivalently,
$$
{\rm Mod}(S)={\rm Homeo}^+(M)/{\rm Homeo}_0(M),
$$
where ${\rm Homeo}_0(M)$ denotes the connected component of the identity in ${\rm Homeo}^+(M)$.
There are several equivalent variations of the definition of the mapping class group, for example, where homeomorphisms are replaced by diffeomorphisms or isotopy classes replaced by homotopy classes. For this and other facts and properties of mapping class groups one can consult~\cite{FM12}. 


\begin{theorem}\cite[Theorems~13.1, 13.2]{FM12}
If $M$ is a closed surface, each class $[f]\in{\rm Mod}(M)$ is either periodic, or reducible, or pseudo-Anosov. 
\end{theorem}

Recall, an element $f\in{\rm Mod}(M)$ is called \emph{periodic} if there is $n\in\N$ such that $f^n$ is isotopic to the identity. An element $f\in{\rm Mod}(M)$ is called \emph{reducible} if there is a non-empty set $\{[c_1],\dots, [c_k]\}$ of isotopy classes of essential simple closed curves in $M$ such that $i([c_i], [c_j])=0,\ i\neq j$, and  $[f(c_i)]=[c_i]$ for all $i=1,2\dots, k$. Here, $[c]$ denotes the homotopy class of a simple closed curve $c$, and $i([a],[b])$ denotes the \emph{geometric intersection number}, i.e., the minimal number of intersection points between a representative curve in the homotopy class $[a]$ and a representative curve in the class $[b]$. 
Finally, an element $f\in{\rm Mod}(M)$ is called  \emph{pseudo-Anosov} if there exists a pair of transverse measured foliations $(\mathcal F^s,\mu_s)$ and $(\mathcal F^u,\mu_u)$ on $M$, the former is called \emph{stable} and the latter \emph{unstable}, a \emph{stretch factor}$\lambda>1$, and a representative \emph{pseudo-Anosov homeomorphism} $\phi$, i.e., 
$$
\phi\cdot(\mathcal F^s,\mu_s)=(\mathcal F^s,\lambda^{-1}\mu_s)\quad{\rm and}\quad \phi\cdot(\mathcal F^u,\mu_u)=(\mathcal F^u,\lambda\mu_u).
$$  
This definition includes Anosov classes on 2-torus $\T^2$.
The difference between Anosov and pseudo-Anosov classes is that transverse measured foliations of the latter may have  singularities. 

\section{Pseudo-Anosov diffeomorphisms}
\label{S:PAD}

\noindent
In this section we exclude the possibility of the map $f$ from Theorem~\ref{T:Main} to be pseudo-Anosov. 

Let $\gamma$ be a homotopically non-trivial closed curve in $M$ and $D$ be a domain in $M$ with contractible in $M$ closure ${\rm Cl}(D)$.  We say that there are $N$ \emph{essential intersections} of $\gamma$ with $D$ if $\gamma\setminus{\rm Cl}(D)$ has $N$ open arcs $\alpha$, and only these arcs, such that $\alpha\cup{\rm Cl}(D)$ is not contractible in $M$. Slightly abusing the terminology, we call the $N$ complementary closed arcs $\beta$ of $\gamma$ to the open arcs $\alpha$  \emph{essential intersections} of $\gamma$ with $D$.
A bound  on the number of essential intersections is needed for the transboundary modulus estimate since the admissibility condition~\eqref{E:Adm} counts intersections with $D_k,\ k\in\Z$, only once.


\begin{lemma}\label{L:EssInt}
Let $M$ be a closed surface of genus at least 1, and 
$S\subset M$ be a compact subset whose complementary components $\{D_k\}_{k\in\Z}$ have contractible in $M$ closures  and such that  $\lim_{|k|\to\infty}{\rm diam}(D_k)=0$. 
Let $L>0$ be fixed and let $\Gamma$ be a family of non-contractible rectifiable closed curves in $M$ such that the length of each $\gamma\in\Gamma$ is at most $L$.
Then there exists $N\in\N$ such that each $\gamma\in\Gamma$ has at most $N$ essential intersections with any $D_k,\ k\in\Z$. Moreover, for any $\gamma$ and any essential intersection $\beta$ of $\gamma$ with $D_k,\ k\in\Z$, the set $\beta\cup D_k$ is contractible in $M$.  
\end{lemma}
\begin{proof}
The assumption that ${\rm Cl}(D_k)$ is contractible in a closed surface $M$ implies that it has an open contractible in $M$ $\eta_k$-neighborhood, $\eta_k>0$.
Thus, for each $\gamma\in\Gamma$ and fixed $k\in\Z$, an open arc $\alpha$, of $\gamma\setminus {\rm Cl}(D_k)$ such that $\alpha\cup{\rm Cl}(D_k)$ is not contractible in $M$ has length at least $\eta_k>0$. 
 Since $M$ is compact and ${\rm diam}(D_k)\to0,\ |k|\to\infty$, the sequence $\eta_k$ may be chosen to have a lower bound $\eta>0$ that is  independent of $k$. Finally, since the length of each $\gamma\in\Gamma$ is at most $L$, the first part of the lemma follows for $N=[L/\eta]$, the integer part of $L/\eta$. 

Now, if $\beta$ is an essential intersection of $\gamma$ with $D_k$, for each open arc $\delta$ of $\beta\setminus{\rm Cl}(D_k),\ k\in\Z$, the set $\delta\cup{\rm Cl}(D_k)$  is contractible in $M$. 
If $\beta\cup D_k$ were not contractible in $M$, there would be an infinite sequence $(\delta_i)_{i\in\N}$ of open arcs of $\beta\setminus{\rm Cl}(D_k)$ that would have to leave the contractible $\eta_k$-neighborhood of $D_k$, and hence the length of $\beta$, and thus of $\gamma$, would be infinite, a contradiction.  	 
\end{proof}

Let $\gamma$ be a simple closed geodesic in $M$, necessarily homotopically non-trivial. 
According to the Collar Theorem, see, e.g., \cite{Bu78}, if genus of $S$ is at least 2, $\gamma$ has a $d$-neighborhood homeomorphic to $\T^1\times(-1,1)$ with $d={\rm arccosh}({\rm coth}(l(\gamma))/2)$, where $l(\gamma)$ is the length of $\gamma$. There is no such an estimate for the torus $\T^2$, however there is still a neighborhood of $\gamma$ in this case since simple closed geodesics on $\T^2$ correspond to straight lines in the plane with rational slopes. In either case, $\gamma$ has an open $d$-neighborhood $U$ and a conformal map $\xi\: U\to A$  to a round annulus
$A=\{a<|z|<b\}$ in the plane. We foliate this annulus with circles $C_t=\{|z|=t\},\ a<t<b$, and let $\gamma_t=\xi^{-1}(C_t),\ a<t<b$, be the corresponding  simple closed curves in the $d$-neighborhood of $\gamma$.  Each $\gamma_t,\  a<t<b$, is necessarily homotopic to $\gamma$. By choosing a larger $a$ and a smaller $b$ we may and will assume that the derivative of $\xi$, with respect to the Euclidean metric in $A$ and the hyperbolic metric in $U$, is uniformly bounded above and below. We denote 
\begin{equation}\label{E:Gam}
\Gamma=\{\gamma_t\: a<t<b\}.
\end{equation}

\begin{corollary}\label{C:NEI}
Let $M$ be a closed surface of genus at least 1 and $f$ a homeomorphism of $M$ that permutes a dense collection $\{D_k\}_{k\in\Z}$ of domains with contractible in $M$ closures and ${\rm diam}(D_k)\to0,\ |k|\to\infty$. Let $\gamma$ be a simple closed geodesic in $M$ and $\Gamma$ a corresponding curve family defined as in~\eqref{E:Gam}.
Then there exists $N\in\N$, such that for each $k, n\in\N$ and each $t,\ a<t<b$, the curve $f^n(\gamma_t)$ has at most $N$ essential intersections with $D_k$.
\end{corollary}
\begin{proof}
Since $\xi\: U\to A$ has derivative bounded below,  there exists $L<\infty$ such that each $\gamma_t\in\Gamma,\ a<t<b$, has length at most $L$. Now the corollary follows from Lemma~\ref{L:EssInt} for $n=0$. 
I.e., there exists $N\in\N$ such that for each $k\in\N$ and each $t,\ a<t<b$, the curve $\gamma_t$ has at most $N$ essential intersections with $D_k$. The map $f^n$ is a homeomorphism for each $n\in\N$, and thus preserves the contractibility property for subsets of $M$, and, in particular, the property of having an essential intersection. Therefore,  the same $N$ also works for each $n\in\N$, and the corollary follows.
\end{proof}


To prove Lemma~\ref{L:PosFin} below we need to discuss the notion and elementary properties of an annular covering surface;  see~\cite[Section~2.2]{St84} for more details.   Let $M$ be a closed surface of genus at least 1 and $\gamma$ a simple closed geodesic in $M$. 
Let $\widetilde M$ be the universal covering surface of $M$, i.e., either the Euclidean or the hyperbolic plane. 
Let $P_0$ be a base point on $\gamma$ and $\widetilde{P_0}$ is one of the preimages of $P_0$ under the covering map $\psi\:\widetilde M\to M$. The geodesic $\gamma$ has a unique lift $\widetilde\gamma$ to $\widetilde M$ that contains $\widetilde{P_0}$ and extends indefinitely in both directions. If $\gamma$ is parametrized by $[0,1]$ with $P_0=\gamma(0)=\gamma(1)$, and $\widetilde\gamma(t),\ t\in[0,1]$, is the lift of $\gamma(t),\ t\in[0,1]$, with $\widetilde\gamma(0)=\widetilde{P_0}$, there exists 
 a covering transformation $T_\gamma$, i.e., an orientation preserving isometry of $\widetilde M$ with $\psi\circ T_\gamma=\psi$, such that  $T_\gamma(\widetilde{P_0})
 =\widetilde\gamma(1)$. The isometry $T_\gamma$ leaves the full lift $\widetilde\gamma$ invariant and
 the action of $T_\gamma$ on $\widetilde\gamma$ is a translation by the length of $\gamma$.   Up to isometry, the quotient $\widehat M_\gamma=\widetilde M/T_\gamma$ does not depend on either the base point $P_0$, or a particular preimage $\widetilde{P_0}$. We refer to such $\widehat M_\gamma$ as the \emph{annular covering surface} associated to $\gamma$. Topologically, $\widehat M_\gamma$ is an open annulus. The curve $\widehat\gamma=\widetilde\gamma/T_\gamma$ is a lift of $\gamma$ to $\widehat M_\gamma$. 
If $\gamma'$ is a closed curve in $M$ that is homotopic to $\gamma$, it has a lift $\widehat{\gamma'}$ to $\widehat M_\gamma$ homotopic to $\widehat\gamma$. 
\begin{lemma}\label{L:NLifts}
Let $\gamma'$ be a closed curve homotopic to a simple closed geodesic $\gamma$ in $M$ and $\widehat\gamma'$ be a lift in $\widehat M_\gamma$ of $\gamma'$ homotopic to $\widehat\gamma$ as above. 
If  $D$ is a domain in $M$ with a contractible in $M$ closure and such that $\gamma'\cap D\neq\emptyset$, then the number of preimages of $D$ in $\widehat M_\gamma$ that intersects $\widehat\gamma'$ is at most the number of essential intersections of $\gamma'$ with $D$. 
\end{lemma}
\begin{proof}
Indeed, any lift to $\widetilde M$ of each essential intersection $\beta'$ of $\gamma'$ with $D$ intersects only one of the lifts of $D$ to $\widetilde M$ as follows from the fact that $\beta'\cup{\rm Cl}(D)$ is contractible in $M$.  Therefore, if $\gamma'$ is parametrized by $[0,1]$, any lift to $\widetilde M$ of the curve $\gamma'(t),\ t\in[0,1]$, that has at most $N$ essential intersections with $D$ intersects  at most $N$ preimages of $D$. Factoring these preimages by the isometry $T_\gamma$ defined above does not increase the number of such preimages in $\widehat M_\gamma$, and the lemma follows. 
\end{proof}
 
 The annular covering surface $\widehat M$ also carries a  metric induced from $\widetilde M$. The curve $\widehat\gamma$ is a homotopically non-trivial simple closed geodesic in $\widehat M_\gamma$ and the length of $\widehat{\gamma'}$ for any $\gamma'$ homotopic to $\gamma$ is at least that of $\widehat{\gamma}$. 
 Moreover, since the covering map $\widehat\psi_\gamma\:\widehat M_\gamma\to M$ is a local isometry, if $\gamma'$ is rectifiable, then so is $\widehat\gamma'$ and they are isometric to each other, in particular have the same lengths. This applies to $\gamma$ and $\widehat\gamma$ as well. Finally, $\widehat M$ supports an orthogonal projection $\Pi_{\widehat\gamma}$ onto $\widehat\gamma$. It is obtained from the orthogonal projection $\Pi_{\tilde\gamma}$ of $\widetilde M$ onto $\tilde\gamma$ using the quotient by $T_\gamma$. It follows from elementary geometry that the orthogonal projection  $\Pi_{\widehat\gamma}$ is well-defined and is a 1-Lipschitz map.     
 
 Let $\{D_k\}_{k\in\Z}$ be a collection of domains with pairwise disjoint closures  in a surface $M$. 
Let $S=M\setminus \cup_{k\in\Z} D_k$.
If $\gamma$ is a curve in $M$, we define its $S$-\emph{length} by 
$$
l_S(\gamma)=\int_{\gamma\cap S}ds+\sum_{\gamma\cap D_k\neq\emptyset} {\rm diam}(D_k).
$$
If $\gamma$ is a locally rectifiable curve in $M$ and $B$ is a Borel measurable subset of $M$, we denote $\int_{\gamma\cap B}ds$ by $l(\gamma\cap B)$.


\begin{lemma}\label{L:PosFin}
Let $M$ be a closed surface of genus at least 1, and $S\subset M$ be a compact subset of $M$ whose complementary components $\{D_k\}_{k\in\Z}$ have bounded geometry. 
For a given $N\in\N$ there exists a constant $C\ge0$ with the following property. 
If $\Gamma$ is a collection of  closed curves $\gamma'$ in the homotopy class of a  simple closed geodesic $\gamma$ such that each $\gamma'\in\Gamma$ has at most $N$ essential intersections with each $D_k,\ k\in\Z$,
then 
$$
{\rm mod}_S(\Gamma)\le C/l(\gamma)^2.
$$
\end{lemma}
\begin{proof}
Let  $\widehat M_\gamma$ be the annular covering surface associated to $\gamma$.  
Let $\gamma'\in\Gamma$ be an arbitrary rectifiable curve homotopic to $\gamma$, $k\in\Z$ be such that 
$D_k$ intersects $\gamma$, and let $\widehat{D_k}$ be a lift of $D_k$ to $\widehat M_\gamma$ intersecting $\widehat\gamma'$, a lift of $\gamma'$ to $\widehat M_\gamma$ that is homotopic to $\widehat\gamma$. Lemma~\ref{L:NLifts} implies that there are at most $N$ such lifts $\widehat{D_k}$. If $D_k,\ k\in\Z$, has a small diameter relative to the injectivity radius of $M$, then the diameter of each lift  $\widehat{D_k}$ is the same as that of $D_k$. If $D_k$ is relatively large, its lift in $\widehat M_\gamma$ may have a larger diameter, but still uniformly comparable to ${\rm diam}(D_k)$. Indeed, the bounded geometry assumption implies that $\lim_{|k|\to\infty}{\rm diam}(D_k)=0$, and hence there are only finitely many domains $D_k$ that are relatively large. The lifts of $D_k,\ k\in\Z$, to the universal cover $\widetilde M$ have diameters that are uniformly comparable to the diameters of the corresponding $D_k$, and the covering $\widetilde M\to \widehat M_\gamma$ does not increase the diameters.   Therefore, there exists a constant $c\in(0,1)$ that depends only on $M$ and $S$ and such that ${\rm diam}(D_k)\ge c\cdot {\rm diam}(\widehat{D_k})$ for all $k\in\Z$ and each lift $\widehat{D_k}$ of $D_k$ to $\widehat M_\gamma$. Let $\widehat S$ be the complement in $\widehat M_\gamma$ of all the lifts, necessarily having disjoint closures, $\widehat{D_k}$ of $D_k,\ k\in\Z$.  
The covering map $\widehat\psi_\gamma\: \widehat M_\gamma\to M$  is a local isometry, and hence $l(\widehat\gamma'\cap\widehat S)=l(\gamma'\cap S)$. Combining the above facts together we obtain that 
\begin{equation}\label{E:Est1}
l_{\widehat S}(\widehat\gamma')\le\frac{N}{c}l_S(\gamma').
\end{equation}
Since $\widehat\gamma'$ is a closed curve homotopic to $\widehat\gamma$, the orthogonal  projections of $\widehat\gamma'\cap\widehat S$ and all $\widehat{D_k}$ that intersect $\widehat\gamma'$, i.e., the images of these sets under $\Pi_{\widehat\gamma}$ onto $\widehat\gamma$, cover $\widehat\gamma$. The projection $\Pi_{\widehat\gamma}$ is a 1-Lipschitz map, and therefore
\begin{equation}\label{E:Est2}
l(\widehat\gamma)\le l_{\widehat S}(\widehat\gamma').
\end{equation}
Combining inequalities~\eqref{E:Est1} and~\eqref{E:Est2} along with  the fact $l(\widehat\gamma)=l(\gamma)$, we obtain
\begin{equation}\label{E:Est3}
l(\gamma)\le\frac{N}{c}l_S(\gamma'),
\end{equation}
for every rectifiable curve $\gamma'$ in $\Gamma$.

We now choose the mass distribution $\rho$ equal 1 on $S$ and $\rho(D_k)={\rm diam}(D_k),\ k\in\Z$. Since $\{D_k\}_{k\in\Z}$ have bounded geometry, $\rho$ has finite mass that depends only on $M$ and $S$ or $\{D_k\}_{k\in\Z}$. 
We now multiply the mass distribution $\rho$ (both its discrete and continuous parts)  by $N/(c\cdot l(\gamma))$. Inequality~\eqref{E:Est3} gives that such new distribution $\rho'$ is admissible for $\Gamma$. 
The total mass of this distribution is 
$$
{\rm mass}(\rho')=\frac{N^2{\rm mass}(\rho)}{c^2 l(\gamma)^2},
$$
and the lemma follows.
\end{proof}

 
To finish excluding the pseudo-Anosov class, we will need the following proposition.
\begin{proposition}\cite[Section~14.5, Theorem~14.23]{FM12}\label{P:ExpGr}
Let $M$ be a closed surface
and suppose that $[f]\in{\rm Mod}(M)$ is pseudo-Anosov with stretch factor $\lambda>1$. 
If $\gamma$ is any homotopically non-trivial  simple closed curve in $M$, then
$$
\lim_{n\to\infty}\sqrt[n]{l[f^n(\gamma)]}=\lambda,
$$ 
where $l[f^n(\gamma)]$ denote the length of the shortest representative in the homotopy class of $f^n(\gamma)$.
\end{proposition}

Now, if $f$ is a pseudo-Anosov homeomorphism satisfying Theorem~\ref{T:Main}, combining Lemmas~\ref{L:Modcomp}, \ref{L:Confinv}, Corollary~\ref{C:NEI}, Lemma~\ref{L:PosFin}, and  Proposition~\ref{P:ExpGr}, we obtain a contradiction. Indeed, let $\gamma$ be an arbitrary simple closed geodesic in $M$ and we define a curve family $\Gamma$ corresponding to $\gamma$ as in~\eqref{E:Gam}. Then ${\rm mod}_2(\Gamma)=\ln(b/a)/(2\pi)$, and hence  by Lemma~\ref{L:Modcomp} (or Corollary~\ref{C:Modcomp0}) and Lemma~\ref{L:PosFin} we have $0<{\rm mod}_S(\Gamma)<\infty$.  Note that Corollary~\ref{C:NEI} implies the assumption on the number of essential intersections of curves from  each $f^n(\Gamma),\ n\in\Z$.
Using Remark~\ref{R:QCC}, we assume that $f$ is quasiconformal on $M$ and conformal on its residual set $S$ which is invariant under $f$. Thus, all iterates of $f$ are also quasiconformal on $M$ and conformal on $S$. Lemma~\ref{L:Confinv} then gives that ${\rm mod}_S(f^n(\Gamma))={\rm mod}_S(\Gamma)$ for all $n\in\Z$, i.e., ${\rm mod}_S(f^n(\Gamma))$ is a positive number independent of $n$. However, Lemma~\ref{L:PosFin} combined with Proposition~\ref{P:ExpGr} imply ${\rm mod}_S(f^n(\Gamma))\to0,\ n\to\infty$, because $l[f^n(\gamma)]\to\infty$, a contradiction.

\section{Holomorphic quadratic differentials}\label{S:HQD}

\noindent
A \emph{holomorphic quadratic differential} $q$ on a Riemann surface $M$  is a holomorphic  section of the symmetric square of the holomorphic cotangent bundle of $M$. In other words,  $q$ is a $(2,0)$-form given in a local chart by $q=q(z)\, dz^2$, where $z\mapsto q(z)$ is a holomorphic function; see~\cite{St84} for background. 

If $q$ is a holomorphic quadratic differential, the locally defined quantity $|dw|=|q(z)|^{1/2}|dz|$ is called \emph{the length element} of the \emph{metric determined by} $q$, or $q$-\emph{metric}. The length of a curve $\gamma$ in this metric will be denoted by $|\gamma|_q$ and called $q$-\emph{length}. The corresponding \emph{area element} is $|q(z)|dx dy$, where $z=x+i y$. 
The $q$-metric is dominated by the intrinsic metric of $M$, i.e., there is a constant $c>0$ such that the length of every curve $\gamma$ in $M$ is at least $c |\gamma|_q$. In particular, diameters of sets in the intrinsic metric are at least $c$ times the diameters in the $q$-metric. 

If $q$ is a holomorphic quadratic differential, its \emph{horizontal trajectory} is a maximal arc on $M$ given in charts by $\arg(q(z)\, dz^2)=0$. A \emph{vertical arc} is an arc in $M$ given by $\arg(q(z)\, dz^2)=\pi$.  
A holomorphic quadratic differential $q$ has \emph{closed trajectories} if its non-closed horizontal trajectories cover a set of measure zero. 
We need the following special case of~\cite[Theorem~1]{Je57}; see~\cite[Theorems~21.1, 21.7]{St84} for similar results, and~\cite{HM79} for a  general result on realizing measured foliations.

\begin{theorem}\cite[Theorem~1]{Je57}\label{T:ExQD}
Let $\gamma$ be a homotopically non-trivial simple closed curve on a closed Riemann surface $M$, and let $\Gamma$ consist of all simple closed curves homotopic to $\gamma$. 
 Then there exists a holomorphic quadratic differential $q$ on $M$ with closed trajectories, so that the metric $|q(z)|^{1/2}|dz|$ is extremal for the 2-modulus of $\Gamma$.  
\end{theorem}

Given a holomorphic quadratic differential $q$ on $M$, one can locally define an analytic function $\Phi(z)=\int\sqrt{q(z)}dz$. Using the function $\Phi$, one can show~\cite[Theorem~9.4]{St84} that each closed horizontal trajectory $\gamma$ of a quadratic differential $q$ can be embedded in a maximal ring domain $R$ swept out by closed horizontal trajectories. It is uniquely determined except  for the case when $M=\T^2$. The function $\Phi^{-1}$ maps a rectangle $(0,a)\times(b_1,b_2)$, for suitably defined $a, b_1, b_2$ conformally onto the maximal ring domain with a vertical arc removed. In the chart given by $\Phi$, the metric $|q(z)|^{1/2}|dz|$ is the Euclidean metric in the rectangle $(0,a)\times(b_1,b_2)$.  
If the points $(0, y)$ and $(a, y),\ b_1<y<b_2$, are identified, the maximal ring domain is conformally equivalent to the straight flat cylinder, which we denote by $\mathcal C_q$, of circumference $a$ and height $b=b_2-b_1$.
The horizontal trajectories are then the horizontal circles on such a cylinder. In particular, closed horizontal trajectories are geodesics in the metric determined by $q$. Note that for the quadratic differential $q$ from Theorem~\ref{T:ExQD}, the length $a$ of the circumference of $\mathcal C_q$ equals 1.


\section{Periodic and reducible diffeomorphisms}\label{S:PR}

\noindent
In this section with exclude the periodic and reducible possibilities for $f$ from Theorem~\ref{T:Main}.

If $M$ is a surface and $f\: M\to M$ a homeomorphism, we say that a curve family $\Gamma$ on $M$ is $f$-\emph{invariant} if $\Gamma=f(\Gamma)$.
If $[f]\in{\rm Mod}(M)$ is either periodic or reducible, then there exists a non-trivial homotopy class $[\gamma]$ of simple closed curves in $M$ that is $f^k$-invariant for some $k\in\N$. By passing to the $k$'th iterate of $f$, we may and will assume that $k=1$. Let $\Gamma$ be the $f$-invariant family of all simple closed curves in $[\gamma]$, or $\Gamma=[\gamma]$.

\begin{lemma}\label{L:NoDiscr}
Let $f$ be a quasiconformal homeomorphism of $M$ that permutes a dense collection of domains $\{D_k\}_{k\in\Z}$ with bounded geometry. Assume further that $f$ is conformal on the residual set $S=M\setminus\cup_{k\in\Z} D_k$. 
Let $\Gamma$ be an $f$-invariant curve family in $M$ with ${\rm mod}_S(\Gamma)<\infty$. If $\rho$ is the extremal mass distribution for ${\rm mod}(\Gamma)$, then $\rho(D_k)=0$ for all $k\in\Z$.
\end{lemma}
\begin{proof}
From Lemma~\ref{L:Confinv} we know that the pullback distribution $f^*(\rho)$ is also extremal for $\Gamma$ because $\Gamma$ is $f$-invariant. Since the extremal distribution is unique, we conclude, in particular, that $\rho(f(D_k))=\rho(D_k)$ for all $k\in\Z$.  Thus, if $D_k, D_l,\ k,l\in\Z$, are in the same $f$-orbit, then $\rho(D_k)=\rho(D_l)$.  Since $f$ permutes $\{D_k\}_{k\in\Z}$, we have for every $k\in\Z$, $f^n(D_k)\cap D_k=\emptyset$ for all $n\neq0$. In particular, each $f$-orbit of domains is infinite. Since ${\rm mod}_S(\Gamma)<\infty$, we conclude that $\rho(D_k)=0$ for all $k\in\Z$.
\end{proof}

We call a mass distribution $\rho$ as in the previous lemma, i.e., when $\rho(D_k)=0$ for all $k\in
\Z$, \emph{purely continuous}. 

The family $\Gamma=[\gamma]$ is $f$-invariant, but it is unclear whether ${\rm mod}_S(\Gamma)<\infty$. By Theorem~\ref{T:ExQD}, there exists a holomorphic  quadratic differential $q$ on $M$ with closed trajectories, so that the metric $|q(z)|^{1/2}|dz|$ is extremal for the 2-modulus of $\Gamma$. Since the $q$-metric is dominated by the intrinsic metric of $M$, we have ${\rm diam}_q(D_k)\le C {\rm diam}(D_k)$ for some $C>0$ and all $k\in\Z$, where ${\rm diam}_q(D_k)$ is the  diameter of $D_k$ measured in the $q$-metric. Thus, 
$$
\sum_{k\in\Z}{\rm diam}_q(D_k)^2<\infty.
$$
Let $\Gamma_h$ consist of all closed horizontal trajectories of $q$ in $\Gamma$. The curve family $\Gamma_h$ has positive 2-modulus but it is not necessarily $f$-invariant. Since each $\gamma_y\in\Gamma_h,\ b_1<y<b_2$, in the above notations, has $q$-length equal 1, Lemma~\ref{L:EssInt} implies that there exists $N\in\N$ such that each such $\gamma_y$ has at most $N$ essential intersections with every $D_k,\ k\in\Z$. Let $\Gamma_N$ consist of all curves in $\Gamma$ that have at most $N$ essential intersections with every $D_k,\ k\in\Z$. Then $\Gamma_h\subset\Gamma_N\subset\Gamma$ and $\Gamma_N$ is $f$-invariant because $f$ is a homeomorphism of $M$. Lemma~\ref{L:PosFin} implies that ${\rm mod}_S(\Gamma_N)<\infty$. Also, by Lemma~\ref{L:Modcomp} (or Corollary~\ref{C:Modcomp0}), ${\rm mod}_S(\Gamma_N)>0$ since ${\rm mod}_2(\Gamma_N)\ge{\rm mod}_2(\Gamma_h)>0$. 
Proposition~\ref{P:Extremaldistr} gives a unique transboundary extremal mass distribution $\rho_0$ for ${\rm mod}_S(\Gamma_N)$. Lemma~\ref{L:NoDiscr} in turn gives that $\rho_0$ is purely continuous. We now show that this is not possible, excluding the periodic and reducible cases. 

\begin{proposition}\label{P:NotPC}
For any homotopically non-trivial simple closed curve $\gamma$ in $M$, 
the extremal mass distribution $\rho_0$ for the curve family $\Gamma_N$ defined above cannot be purely continuous.
\end{proposition}
\begin{proof} 
Suppose that $\rho_0$ is
purely continuous.
Let $\rho_1$ be the mass distribution for $\Gamma_N$ that is given in local charts as $|q(z)|^{1/2}$ on $S$ and $\rho_1(D_k)=(N/c)\, {\rm diam}_q(D_k),\ k\in\Z$, where $c$ is a constant as in the proof of Lemma~\ref{L:PosFin}. As in Lemma~\ref{L:PosFin}, this distribution is admissible for $\Gamma_N$, it has finite total mass, but it is not necessarily extremal. To finish the proof, we need the following lemma.

\begin{lemma}\label{L:CoPa}
The $L^2$-norm of the continuous part of the mass distribution $\rho_1$ is strictly less than $\|\rho_0\|_{L^2}=\|\rho_0\|_{L^2\oplus l^2}$. 
\end{lemma}
\begin{proof}
First, $\|\rho_1\|_{L^2}={\rm area}_q(S)$, the area of $S$ in the $q$-metric. Since $D_k,\ k\in\Z$, have interior, it is strictly less than the $q$-area of the whole cylinder $\mathcal C_q$, which is the 2-modulus of the family $\Gamma_h$ of all closed horizontal trajectories of $q$ by Theorem~\ref{T:ExQD}.  Since each closed horizontal trajectory $\gamma_y$ of $q$ is homotopic to $\gamma$, integrating $\rho_0$ over each such trajectory gives
$$
\int_{\gamma_y}\rho_0\,  ds=\int_{\gamma_y\cap S}\rho_0\,  ds\ge1,
$$
for all $y,\ b_1<y<b_2$. This means that $\rho_0|_S$ is admissible for $\Gamma_h$. Therefore, 
$$\|\rho_0\|_{L^2}={\rm mass}(\rho_0)\ge{\rm mod}_2(\Gamma_h)={\rm area}_q(\mathcal C_q)>{\rm mass}(\rho_1|_S)=\|\rho_1\|_{L^2}.
$$
\end{proof}

Now, consider mass distributions
$$
\rho_t=(1-t)\rho_0+t\rho_1,\quad 0\le t\le1. 
$$
This is a one-parameter family of mass distributions connecting $\rho_1$ to $\rho_0$. Convexity gives that each of this distributions is admissible for $\Gamma_N$. Strict convexity of the ball of radius ${\rm mass}(\rho_0)$ in $L^2\oplus l^2$ along with Lemma~\ref{L:CoPa} imply that there exists $t$ close to 0 such that ${\rm mass}(\rho_t)<{\rm mass}(\rho_0)$.
%
Indeed, 
$$
{\rm mass}(\rho_t)-{\rm mass}(\rho_0)\sim 2t\int_S\rho_0(\rho_1-\rho_0)d\sigma,\quad t\to0.
$$
But
$$
\int_S\rho_0\rho_1\, d\sigma\le\left({\rm mass}(\rho_0)\right)^{1/2}\left({\rm mass}(\rho_1|_S)\right)^{1/2}<{\rm mass}(\rho_0)=\int_S\rho_0^2d\sigma,
$$
and so $
{\rm mass}(\rho_t)<{\rm mass}(\rho_0)$ for small $t>0$.
This contradicts the assumption that $\rho_0$ is extremal for $\Gamma_N$ and concludes the proof of Proposition~\ref{P:NotPC}.
\end{proof}

\end{document}